\documentclass{article}
\usepackage{epsfig,color,pstricks,pst-poly}
\usepackage{graphicx,amssymb,amsmath,theorem}
\usepackage{amscd,latexsym}
\usepackage{tikz}
\usepackage{subfigure}
\usepackage{hyperref}

\long\def\killtext#1{}

\newtheorem{theorem}{Theorem}[section]

\newtheorem{definition}[theorem]{Definition}
\newtheorem{lemma}[theorem]{Lemma}

\newtheorem{corollary}[theorem]{Corollary}

\newtheorem{conjecture}[theorem]{Conjecture}
\theorembodyfont{\rmfamily}

\newenvironment{proof}{\medskip\noindent{\bf Proof. }}{\hfill$\square$\medskip}

\begin{document}
\title{\bf Drawing cubic graphs with the four basic slopes}

\author{{\sc Padmini Mukkamala} and {\sc D\"om\"ot\"or P\'alv\"olgyi}\thanks{
The second author was supported by the European Union and co-financed by the European Social Fund (grant agreement no. TAMOP 4.2.1/B-09/1/KMR-2010-0003).
Part of this work was done in Lausanne and the authors gratefully acknowledge support
from the Bernoulli Center at EPFL and from the Swiss National Science
Foundation, Grant No. 200021-125287/1.}\\
Rutgers, the State University of New Jersey\\
E\"otv\"os University, Budapest 
}

\date{}

\maketitle %\tableofcontents

\begin{abstract}
We show that every cubic graph can be drawn in the plane with straight-line edges using only the four basic slopes $\{0,\pi/4,\pi/2,3\pi/4\}$. We also prove that four slopes have this property if and only if we can draw $K_4$ with them.
\end{abstract}

\section{Introduction}
A drawing of a graph is said to be a {\em straight-line drawing} if the vertices of $G$ are represented by distinct points in the plane and every edge is represented by a straight-line segment connecting the corresponding pair of vertices and not passing through any other vertex of $G$.
If it leads to no confusion, in notation and terminology we make no distinction between a vertex and the corresponding point, and between an edge and the corresponding segment.
The {\em slope} of an edge in a straight-line drawing is the slope of the corresponding segment.
Wade and Chu \cite{wc94} defined the {\em slope number}, $sl(G)$, of a graph $G$ as the smallest number $s$ with the property that $G$ has a straight-line drawing with edges
of at most $s$ distinct slopes.

Obviously, if $G$ has a vertex of degree $d$, then its slope number is at least
$\lceil d/2\rceil$. Dujmovi\'c et al.~\cite{dsw04} asked if the slope number of a graph with bounded maximum degree $d$ could be arbitrarily large. Pach and P\'alv\"olgyi \cite{pp06} and Bar\'at, Matou\v sek, Wood \cite{bmw06} (independently) showed with a counting argument that the answer is no for $d\ge 5$.

In \cite{kppt08_2}, it was shown that cubic ($3$-regular) graphs could be drawn with five slopes. The major result from which this was concluded was that subcubic graphs\footnote{A graph is subcubic if it is a proper subgraph of a cubic graph, i.e. the degree of every vertex is at most three and it is not cubic (not $3$-regular).} can be drawn with the four basic slopes, the slopes $\{0,\pi/4,\pi/2, 3\pi/4\}$, corresponding to the vertical, horizontal and the two diagonal directions.

This was improved in \cite{ms} to show that connected cubic graphs can be drawn 
with four slopes\footnote{But not the four basic slopes.} while disconnected cubic graphs required five slopes.

It was shown by Max Engelstein \cite{eng} that $3$-connected cubic graphs with a Hamiltonian cycle can be drawn with the four basic slopes.

We improve all these results by the following

\begin{theorem}\label{mainthm}
Every cubic graph has a straight-line drawing with only the four basic slopes.
\end{theorem}

% Figure by Padmini
\begin{figure*}[htp]
{\centering
\subfigure[Petersen graph]{
\begin{tikzpicture}[scale=1]
%\filldraw [blue!80!black!20!white] (0,0) -- (1,1) -- (2.5,1) -- (3.5,0) -- (2.5,-1) -- (1,-1) -- (0,0) -- cycle;
\node [fill=black,circle,inner sep=1pt] (1) at (1,0) {}; 
\node [fill=black,circle,inner sep=1pt] (2) at (0,1) {}; 
\node [fill=black,circle,inner sep=1pt] (3) at (2,1) {}; 
\node [fill=black,circle,inner sep=1pt] (4) at (4,1) {}; 
\node [fill=black,circle,inner sep=1pt] (5) at (3,0) {}; 
\node [fill=black,circle,inner sep=1pt] (6) at (0,2) {}; 
\node [fill=black,circle,inner sep=1pt] (7) at (1,3) {}; 
\node [fill=black,circle,inner sep=1pt] (8) at (2,3) {}; 
\node [fill=black,circle,inner sep=1pt] (9) at (3,3) {}; 
\node [fill=black,circle,inner sep=1pt] (10) at (4,2) {}; 
\draw [black] (1) -- (2) -- (3) -- (4) -- (5) -- (1) -- (7) -- (8) -- (9) -- (10) -- (6) -- (7);
\draw [black] (2) -- (6);
\draw [black] (3) -- (8);
\draw [black] (4) -- (10);
\draw [black] (5) -- (9);
%\draw [->,green] (-0.5,4) arc (145:35:22pt);
\end{tikzpicture}
}
\qquad
\subfigure[$K_{3,3}$]{
\begin{tikzpicture}[scale=1]
\node [fill=black,circle,inner sep=1pt] (1) at (1,0) {}; 
\node [fill=black,circle,inner sep=1pt] (2) at (3,0) {}; 
\node [fill=black,circle,inner sep=1pt] (3) at (4,1) {}; 
\node [fill=black,circle,inner sep=1pt] (4) at (3,2) {}; 
\node [fill=black,circle,inner sep=1pt] (5) at (1,2) {}; 
\node [fill=black,circle,inner sep=1pt] (6) at (0,1) {}; 
\draw [black] (1) -- (2) -- (3) -- (4) -- (5) -- (6) -- (1);
\draw [black] (1) -- (4);
\draw [black] (2) -- (5);
\draw [black] (3) -- (6);
\end{tikzpicture}
}
\caption{The Petersen graph and $K_{3,3}$ with the four basic slopes.}
} 
\label{fig:petersen}
\end{figure*}
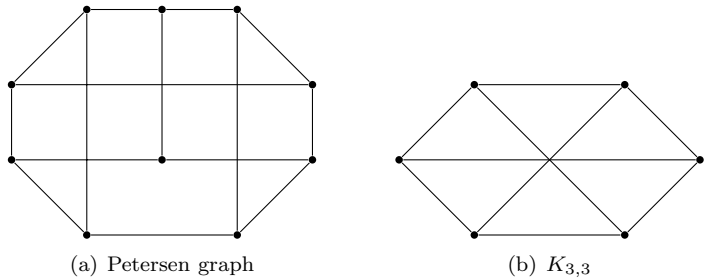

This is the first result about cubic graphs that uses a nice, fixed set of slopes instead of an unpredictable set, possibly containing slopes that are not rational multiples of $\pi$.
Also, since $K_4$ requires at least $4$ slopes, this settles the question of determining the minimum number of slopes required for cubic graphs.
In the last section we also prove

\begin{theorem}\label{karakterizacio}
Call a set of slopes {\em good} if every cubic graph has a straight-line drawing with them. 
Then the following statements are equivalent for a set $S$ of four slopes.
\begin{enumerate}
\item $S$ is good.
\item $S$ is an affine image of the four basic slopes.
\item We can draw $K_4$ with $S$.
\end{enumerate}
\end{theorem}

The problem whether the slope number of graphs with maximum degree four is unbounded or not remains an interesting open problem.

%\begin{conjecture} The slope number of graphs with maximum degree $4$ is unbounded.
%\end{conjecture}

There are many other related graph parameters. 
The {\em thickness} of a graph $G$ is defined as the smallest number of planar subgraphs it can be decomposed into \cite{MuOS}. It is one of the several widely known graph parameters that measures how far $G$ is from being planar.
The {\em geometric thickness} of $G$, defined as the smallest number of {\em crossing-free} subgraphs of a straight-line drawing of $G$ whose union is $G$, is another similar notion \cite{Ka}.
It follows directly from the definitions that the thickness of any graph is at most as
large as its geometric thickness, which, in turn, cannot exceed its slope number.
For many interesting results about these parameters, consult \cite{DiEH, dek04, dsw04, DuW, E04, HuSW}.

A variation of the problem arises if (a) two vertices in a drawing have an edge between them if and only if the slope between them belongs to a certain set $S$ and, (b) collinearity of points is allowed. This violates the condition stated before that an edge cannot pass through vertices other than its end points. For instance, $K_n$ can be drawn with one slope. The smallest number of slopes that can be used to represent a graph in such a way is called the {\em slope parameter} of the graph.
Under these set of conditions, \cite{ambaha06} proves that the slope parameter of subcubic outerplanar graphs is at most $3$.
It was shown in \cite{kppt10} that the slope parameter of every cubic graph is at most seven.
If only the four basic slopes are used, then the graphs drawn with the above conditions are called queens graphs and \cite{amba06} characterizes certain graphs as queens graphs. Graph theoretic properties of some specific queens graphs can be found in \cite{bs09}.

Another variation for planar graphs is to demand a planar drawing. The {\em planar slope number} of a planar graph is the smallest number of distinct slopes with the property that the graph has a straight-line drawing with non-crossing edges using only these slopes.
Dujmo\-vi\'c, Eppstein, Suderman, and Wood \cite{dsw07} raised the question whether there exists a function $f$ with the property that the planar slope number of every planar graph with maximum degree $d$ can be bounded from above by $f(d)$.
Jelinek et al.~\cite{JJ10} have shown that the answer is yes for {\em outerplanar} graphs, that is, for planar graphs that can be drawn so that all of their vertices lie on the outer face.
Eventually the question was answered in \cite{kpp10} where it was proved that any bounded degree planar graph has a bounded planar slope number.

Finally we would mention a slightly related problem. Didimo et al.~\cite{Didimo} studied drawings of graphs where edges can only cross each other in a right angle. Such a drawing is called an RAC (right angle crossing) drawing. They showed that every graph has an RAC drawing if every edge is a polygonal line with at most three bends (i.e. it consists of at most four segments). They also gave upper bounds for the maximum number of edges if less bends are allowed. Later Arikushi et al.~\cite{rado} showed that such graphs can have at most $O(n)$ edges. Angelini et al.~\cite{Angelini} proved that every cubic graph admits an RAC drawing with at most one bend. It remained an open problem whether every cubic graph has an RAC drawing with straight-line segments. If besides orthogonal crossings, we also allow two edges to cross at $45^\circ$, then it is a straightforward corollary of Theorem \ref{mainthm} that every cubic graph admits such a drawing with straight-line segments.

In section $2$ we give the proof of the Theorem \ref{mainthm} while in section $3$ we prove Theorem \ref{karakterizacio} and discuss open problems.

% Figure by Padmini
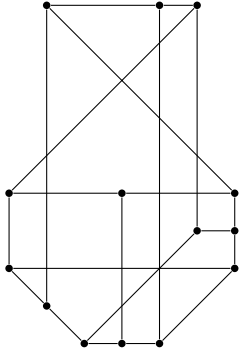
\begin{figure*}[htp]
{\centering
\begin{tikzpicture}[scale=0.5]
%\filldraw [blue!80!black!20!white] (0,0) -- (1,1) -- (2.5,1) -- (3.5,0) -- (2.5,-1) -- (1,-1) -- (0,0) -- cycle;
\node [fill=black,circle,inner sep=1pt] (1) at (0,2) {}; 
\node [fill=black,circle,inner sep=1pt] (2) at (1,1) {}; 
\node [fill=black,circle,inner sep=1pt] (3) at (2,0) {}; 
\node [fill=black,circle,inner sep=1pt] (4) at (3,0) {}; 
\node [fill=black,circle,inner sep=1pt] (5) at (4,0) {}; 
\node [fill=black,circle,inner sep=1pt] (6) at (6,2) {}; 
\node [fill=black,circle,inner sep=1pt] (7) at (6,3) {}; 
\node [fill=black,circle,inner sep=1pt] (8) at (5,3) {}; 
\node [fill=black,circle,inner sep=1pt] (9) at (0,4) {}; 
\node [fill=black,circle,inner sep=1pt] (10) at (3,4) {}; 
\node [fill=black,circle,inner sep=1pt] (11) at (6,4) {}; 
\node [fill=black,circle,inner sep=1pt] (12) at (1,9) {}; 
\node [fill=black,circle,inner sep=1pt] (13) at (4,9) {}; 
\node [fill=black,circle,inner sep=1pt] (14) at (5,9) {}; 
\draw [black] (1) -- (2) -- (3) -- (4) -- (5) -- (6) -- (1);
\draw [black] (3) -- (8) -- (7) -- (6);
\draw [black] (9) -- (10) -- (11) -- (12) -- (13) -- (14) -- (9);
\draw [black] (1) -- (9);
\draw [black] (2) -- (12);
\draw [black] (4) -- (10);
\draw [black] (5) -- (13);
\draw [black] (7) -- (11);
\draw [black] (8) -- (14);
%\draw [->,green] (-0.5,4) arc (145:35:22pt);
\end{tikzpicture}
\caption{The Heawood graph drawn with the four basic slopes.}
} 
\label{fig:heawood}
\end{figure*}

\section{Proof of Theorem \ref{mainthm}}

We start with some definitions we will use throughout the section.

\subsection{Definitions and Subcubic Theorem}
\indent \vspace{-0.3cm}

Throughout the paper $\log$ always denotes $\log_2$, the logarithm in base $2$.

We recall that the girth of a graph is the length of its shortest cycle.

\begin{definition}
Define a {\em  supercycle} as a connected graph where every degree is at least two and not all are two. Note that a minimal supercycle will look like a ``$\theta$'' or like a ``dumbbell''. 
\end{definition}

We recall that a {\em cut} is a partition of the vertices into two sets. We say that an edge is in the cut if its ends are in different subsets of the partition. We also call the edges in the cut the {\em cut-edges}. The {\em size} of a cut is the number of cut-edges in it.

\begin{definition}
We say that a cut is an {\em $M$-cut} if the cut-edges form a matching, in other words, if their ends are pairwise different vertices.
We also say that an {\em $M$-cut} is suitable if after deleting the cut-edges, the graph has two components, both of which are supercycles.
\end{definition}

For any two points $p_1=(x_1,y_1)$ and $p_2=(x_2,y_2)$, we
say that $p_2$ is {\em to the North} of $p_1$ if
$x_2=x_1$ and $y_2>y_1$ . Analogously, we say that
$p_2$ is {\em to the Northwest} of $p_1$ if $x_2+y_2=x_1+y_1$ and $y_2>y_1$.

We will give the exact statement of the theorem of \cite{kppt08_2} about subcubic graphs here as it will be used in this proof.

\begin{theorem}[\cite{kppt08_2}]\label{pachetal}
Let $G$ be a connected graph that is not a cycle and whose every vertex has degree at most three. Suppose that $G$ has at least one vertex of degree at most two and denote by $v_1,\ldots,v_m$ the vertices of degree at most two ($m \ge 1$). 

Then, for any sequence $x_1,\ldots,x_m$ of real numbers, linearly independent over the rationals, $G$ has a straight-line drawing with the following properties:

(1) Vertex $v_i$ is mapped into a  point with $x$-coordinate $x(v_i) = x_i \ \ \ (1 \le i \le m)$

(2) The slope of every edge is $0, \pi/2, \pi/4,$ or $-\pi/4$

(3) No vertex is to the North of any vertex of degree two.

(4) No vertex is to the North or to the Northwest of any vertex of degree one.

\end{theorem}

It seems that the proof of the theorem about subcubic graphs in \cite{kppt08_2} was slightly incorrect. It used induction but during the proof the statement was also used for disconnected graphs.
This can be a problem, as when drawing two components, it might happen that a degree three vertex of one component has to be above a degree two vertex of the other component.
However, the proof can be easily fixed to hold for disconnected graphs as well.
For this, one can make the statement stronger, by saying that also for every graph one can select any sequence $x_{m+1},\ldots,x_n$ of real numbers that satisfy that $x_1,\ldots,x_m, x_{m+1},\ldots,x_n$ are linearly independent over the rationals, such that the $x$-coordinates of all the vertices are a linear combination with rational coefficients of $x_1,\ldots,x_n$.
This way we can ensure that different components do not interfere.\\

Note that Theorem \ref{pachetal} proves the result of Theorem \ref{mainthm} for subcubic graphs. 
Another minor observation is that we may assume that the graph is connected. 
Since we use the basic four slopes, if we can draw the components of a 
disconnected graph, then we just place them far apart in the plane so 
that no two drawings intersect. 
So we will assume for the rest of the section that the graph is cubic and connected.

\subsection{Preliminaries}
The results in this subsection are also interesting independent of the current problem we deal with.

\begin{lemma}\label{girth}
Every connected cubic graph on $n$ vertices contains a cycle of length at most 
$2 \lceil \log ( \frac{n}{3} +1) \rceil $.
\end{lemma}

% Figure by Padmini
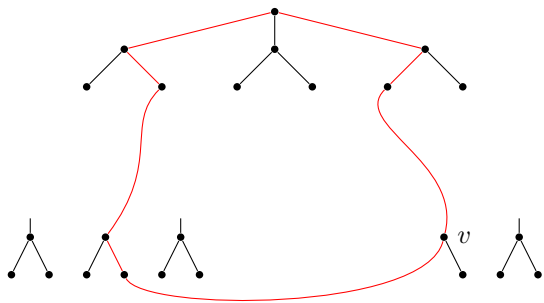
\begin{figure*}[h]
{\centering
\begin{tikzpicture}[scale=0.5]
%\filldraw [blue!80!black!20!white] (0,0) -- (1,1) -- (2.5,1) -- (3.5,0) -- (2.5,-1) -- (1,-1) -- (0,0) -- cycle;
\node [fill=black,circle,inner sep=1pt] (1) at (7,7) {}; 
\node [fill=black,circle,inner sep=1pt] (2) at (3,6) {}; 
\node [fill=black,circle,inner sep=1pt] (3) at (7,6) {}; 
\node [fill=black,circle,inner sep=1pt] (4) at (11,6) {}; 
\node [fill=black,circle,inner sep=1pt] (5) at (2,5) {}; 
\node [fill=black,circle,inner sep=1pt] (6) at (4,5) {}; 
\node [fill=black,circle,inner sep=1pt] (7) at (6,5) {}; 
\node [fill=black,circle,inner sep=1pt] (8) at (8,5) {}; 
\node [fill=black,circle,inner sep=1pt] (9) at (10,5) {}; 
\node [fill=black,circle,inner sep=1pt] (10) at (12,5) {}; 
\node [fill=black,circle,inner sep=1pt] (11) at (0,0) {}; 
\node [fill=black,circle,inner sep=1pt] (12) at (0.5,1) {}; 
\node [fill=black,circle,inner sep=1pt] (13) at (1,0) {}; 
\node [fill=black,circle,inner sep=1pt] (14) at (2,0) {}; 
\node [fill=black,circle,inner sep=1pt] (15) at (2.5,1) {}; 
\node [fill=black,circle,inner sep=1pt] (16) at (3,0) {}; 
\node [fill=black,circle,inner sep=1pt] (17) at (4,0) {}; 
\node [fill=black,circle,inner sep=1pt] (18) at (4.5,1) {}; 
\node [fill=black,circle,inner sep=1pt] (19) at (5,0) {}; 
\node (20) at (11,-1) {}; 
\node (26) at (3.5,-1) {}; 
\node [fill=black,circle,inner sep=1pt,label=right:$v$] (21) at (11.5,1) {}; 
\node [fill=black,circle,inner sep=1pt] (22) at (12,0) {}; 
\node [fill=black,circle,inner sep=1pt] (23) at (13,0) {}; 
\node [fill=black,circle,inner sep=1pt] (24) at (13.5,1) {}; 
\node [fill=black,circle,inner sep=1pt] (25) at (14,0) {}; 
\draw [black] (2) -- (5);
\draw [black] (4) -- (10);
\draw [red] (2) -- (1) -- (4);
\draw [black] (1) -- (3) -- (7);
\draw [red] (2) -- (6);
\draw [black] (8) -- (3);
\draw [red] (9) -- (4);

\draw [black] (11) -- (12) -- (13);
\draw [black] (12) -- (0.5,1.5);
\draw [black] (18) -- (4.5,1.5);
\draw [black] (24) -- (13.5,1.5);
\draw [black] (14) -- (15);
\draw [red] (16) -- (15);
\draw [black] (17) -- (18) -- (19);
\draw [black] (23) -- (24) -- (25);
\draw [black] (21) -- (22);
\draw [red] (21) .. controls (20) and (26) .. (16);
\draw [red] (6) .. controls (3,4) and (4,3) .. (15);
\draw [red] (9) .. controls (9,4) and (12,3) .. (21);

%\draw [->,green] (-0.5,4) arc (145:35:22pt);
\end{tikzpicture}
\caption{Finding a cycle in the BFS tree using that the left child of $v$ already occurred. %before and the two paths to the root give the required cycle.
}
} 
\label{fig:cycle}
\end{figure*}

\begin{proof}
Start at any vertex of $G$ and conduct a breadth first search (BFS) of 
$G$ until a vertex repeats in the BFS tree. 
We note here that by iterations we will (for the rest of the subsection)
mean the number of levels of the BFS tree.
Since $G$ is cubic, after $k$ iterations, the number of vertices visited
will be $1 + 3+ 6+ 12 + \ldots +3 \cdot 2^{k-2} = 1+3(2^{k-1}-1)$. 
And since $G$ has $n$ vertices, some vertex must repeat after 
$k = \lceil \log (\frac{n}{3} + 1) \rceil +1$ iterations.
Tracing back along the two paths obtained for the vertex that reoccurs,
we find a cycle of length at most
$2 \lceil \log ( \frac{n}{3} +1) \rceil $.
%We must be careful though that $2 \lceil \log ( \frac{n}{3} +1) \rceil \le n$.
%and we can check that this inequality holds for all values of $n \ge 4$.
\end{proof}

\begin{lemma}\label{supercycle}
Every connected cubic graph on $n$ vertices with girth $g$ contains a supercycle with at most $2 \lceil \log (\frac{n-1}{g}) \rceil +g-1$ vertices.
\end{lemma}
\begin{proof}
Contract the vertices of a length $g$ cycle, obtaining a multigraph $G'$ with $n-g+1$ vertices, that is almost $3$-regular, except for one vertex of degree $g$, from which we start a BFS. It is easy to see that the number of vertices visited after $k$ iterations is at most $1 + g + 2g + 4g + \ldots + g\cdot 2^{k-2} = g(2^{k-1} -1)+1$. 
And since $G'$ has $n-g+1$ vertices, some vertex must repeat after 
$k = \lceil \log (\frac{n-g+1}{g} + 1) \rceil +1=\lceil \log (\frac{n+1}{g}) \rceil +1$ iterations.
Tracing back along the two paths obtained for the vertex that reoccurs,
we find a cycle (or two vertices connected by two edges) of length at most $2 \lceil \log (\frac{n-1}{g}) \rceil$ in $G'$.
This implies that in $G$ we have a supercycle with at most $2 \lceil \log (\frac{n-1}{g}) \rceil +g-1$ vertices.
\end{proof}

\begin{lemma}\label{Mcut}
Every connected cubic graph on $n>2s -2$ vertices with a supercycle with $s$ vertices contains a suitable $M$-cut of size at most $s -2$.
\end{lemma}
\begin{proof}
The supercycle with $s$ vertices, $A$, has at least two vertices of degree $3$.
The size of the $(A, G-A)$ cut is thus at most $s -2$. 
This cut need not be an $M$-cut because the edges may have a common neighbor
in $G-A$. To repair this, we will now add, iteratively, 
the common neighbors of edges in the cut to $A$, 
until no edges have a common neighbor in $G-A$.
Note that in any iteration, if a vertex, $v$, adjacent to exactly two cut-edges
was chosen, then the size of $A$ increases by $1$ and the size of the cut decreases
by $1$ (since, these two cut-edges will get added to $A$ along with $v$,  but
 since the graph is cubic, the third edge from $v$ will become a part of the
cut-edges).
If a vertex adjacent to three cut-edges was chosen, then the size of $A$ increases
by $1$ while the number of cut-edges decreases by $3$.
From this we can see that the maximum number of vertices that could
have been added to $A$ during this process is $s -3$. 
Now there are three conditions to check.

The first condition is that this process returns
a non-empty second component. This would occur if 
$$ (n- s) - (s-3)>0$$
or,
$$ n > 2s -3.$$

The second condition is that the second component should not be a collection of disjoint cycles.
For this we note that it is enough to check that at every stage, the 
number of cut-edges is strictly smaller than the number of vertices
in $G-A$. But since in the above iterations, the number of cut-edges decreases
by a number greater than or equal to the decrease in the size of $G-A$, it is enough to 
check that before the iterations, the number of cut-edges is strictly 
smaller than the number of vertices in $G-A$.
This is the condition
$$n - s > s-2$$
or,
$$ n>2s -2.$$

Note that if this inequality holds then the non-emptiness condition will 
also hold.

Finally, we need to check that both components are connected. $A$ is connected
but $G-A$ need not be. But this last step is the easiest. We pick a component 
in $G-A$ that has more vertices than the number of cut-edges adjacent to it.
Since the number of cut-edges is strictly smaller than number of vertices in
$G-A$, there must be one such component, say $B$, in $G-A$. We add every other
component of $G-A$ to $A$. Note that the size of the cut only decreases with
this step. Since $B$ is connected and has more vertices than the number of
cut-edges, $B$ cannot be a cycle.
\end{proof}

\begin{corollary}\label{exisMcut}
Every connected cubic graph on $n \ge 18$ vertices contains a suitable $M$-cut.
\end{corollary}
\begin{proof}
Using the first two lemmas, we have a supercycle with $s\le 2 \lceil \log (\frac{n+1}{g}) \rceil +g-1$ vertices where $3\le g\le 2\lceil \log ( \frac{n}{3} +1) \rceil $. Then using the last lemma, we have an $M$-cut with both partitions being a supercycle if $n>2s-2$. So all we need to check is that $n$ is indeed big enough.
Note that

$$s\le 2 \log (\frac{n+1}{g}) +g+1= 2 \log (n+1) + g-2\log g
\le 2 \log (n+1) + 2 \log ( \frac{n}{3} +1) -2\log (2 \log ( \frac{n}{3} +1))+1$$

where the last inequality follows from the fact that $x-2\log_e x$ is increasing for $x\ge 2/\log 2\approx 2.88$. So we can bound the right hand side from above by $4 \log (n+1) +1$. Now we need that 

$$n> 2(4 \log (n+1) +1)-2=8\log (n+1)$$

which holds if $n\ge 44$.

The statement can be checked for $18\le n\le 42$ with a code that can be found in the Appendix. It outputs for a given value of $n$, the $g$ for which $2s-2$ is maximum and this maximum value.
Based on the output we can see that for $n \ge 18$, this value is smaller.
\end{proof}

\subsection{Proof}

\begin{lemma}\label{cubicdr}
Let $G$ be a connected cubic graph with a suitable $M$-cut. % that disconnects $G$ into two connected components that are not cycles.
Then, $G$ can be drawn with the four basic slopes.
\end{lemma}

% Figure by Padmini
\begin{figure*}[h]
{\centering
\begin{tikzpicture}[scale=0.55]
%\filldraw [blue!80!black!20!white] (0,0) -- (1,1) -- (2.5,1) -- (3.5,0) -- (2.5,-1) -- (1,-1) -- (0,0) -- cycle;
\node [fill=black,circle,inner sep=1pt] (1) at (1,1) {}; 
\node [fill=black,circle,inner sep=1pt] (2) at (2,2) {}; 
\node [fill=black,circle,inner sep=1pt] (3) at (5,1) {}; 
\node [fill=black,circle,inner sep=1pt] (4) at (6,2) {}; 
\node [fill=black,circle,inner sep=1pt] (5) at (7,1.5) {}; 
\node [fill=black,circle,inner sep=1pt] (6) at (11,1) {}; 
\node [fill=black,circle,inner sep=1pt] (7) at (12,2) {}; 
\node [fill=black,circle,inner sep=1pt] (8) at (13,2) {}; 
\node [fill=black,circle,inner sep=1pt] (9) at (16,1) {}; 
\node [fill=black,circle,inner sep=1pt] (10) at (17,2) {}; 
\node [fill=black,circle,inner sep=1pt] (11) at (11,9) {}; 
\node [fill=black,circle,inner sep=1pt] (12) at (12,8) {}; 
\node [fill=black,circle,inner sep=1pt] (13) at (13,9) {}; 
\node [fill=black,circle,inner sep=1pt] (14) at (16,8) {}; 
\node [fill=black,circle,inner sep=1pt] (15) at (17,8.5) {}; 

\draw [black] (8,1.35) arc (0:360:4cm and 1.2cm);
\draw [black] (18.2,1.35) arc (0:360:4cm and 1.2cm);
\draw [black] (18.2,8.5) arc (0:360:4cm and 1.2cm);

\draw [dashed,black] (11,-1.9) node[below] {$x_1$} -- (6);
\draw [dashed,black] (12,-1) node[below] {$x_2$} -- (7);
\draw [dashed,black] (13,-1.9) node[below] {$x_3$} -- (8);
\draw [dashed,black] (16,-1) node[below] {$x_{m-1}$} -- (9);
\draw [dashed,black] (17,-1.9) node[below] {$x_m$} -- (10);
\draw [dashed,black] (1,-1.9) node[below] {$-x_m$} -- (1);
\draw [dashed,black] (2,-1) node[below] {$-x_{m-1}$} -- (2);
\draw [dashed,black] (5,-1.9) node[below] {$-x_3$} -- (3);
\draw [dashed,black] (6,-1) node[below] {$-x_2$} -- (4);
\draw [dashed,black] (7,-1.9) node[below] {$-x_1$} -- (5);
\draw [black] (6) -- (11);
\draw [black] (7) -- (12);
\draw [black] (8) -- (13);
\draw [black] (9) -- (14);
\draw [black] (10) -- (15);

\draw [very thick,double,->] (4,3) node[above=3cm] {Rotated and translated} arc (180:90:5cm);
%\draw [red] (21) .. controls (20) and (26) .. (16);
%\draw [red] (6) .. controls (3,4) and (4,3) .. (15);
%\draw [red] (9) .. controls (9,4) and (12,3) .. (21);

%\draw [->,green] (-0.5,4) arc (145:35:22pt);
\end{tikzpicture}
\caption{The $x$-coordinates of the degree $2$ vertices is suitably chosen and one component is rotated and translated to make the $M$-cut vertical.}
} 
\label{fig:final}
\end{figure*}
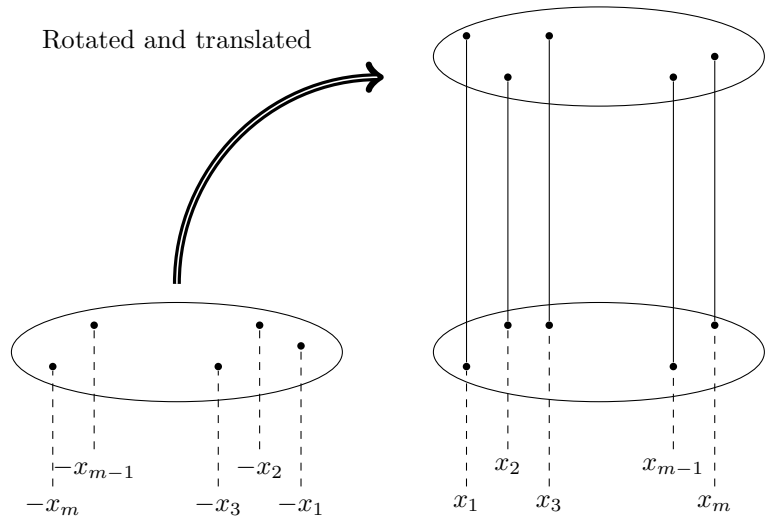

\begin{proof}
The proof follows rather straightforwardly from \ref{pachetal}. 
Note that the two components are subcubic graphs and we can choose 
the $x$-coordinates of the vertices of the $M$-cut (since they are 
the vertices with degree two in the components). If we picked coordinates 
$x_1, x_2, \ldots, x_m$ in one component, then for the neighbors of these vertices in the other component we pick the $x$-coordinates $-x_1,-x_{2},\ldots, -x_m$. We now rotate the second
component by $\pi$ and place it very high above the other component so that the drawings 
of the components do not intersect
and align them so that the edges of the $M$-cut will be vertical (slope $\pi/2$). 
Also, since Theorem \ref{pachetal} guarantees that degree two vertices have 
no other vertices on the vertical line above them, hence the drawing
we obtain above is a valid representation of $G$ with the basic slopes.
\end{proof}

From combining Lemma \ref{exisMcut} and Lemma \ref{cubicdr}, we can see that Theorem \ref{mainthm} 
is true for all cubic graphs with $n\ge 18$. For smaller graphs, we give 
below some lemmas which help reduce the number of graphs we have to check.
The lemmas below also occur in different papers and we give
references where required.

\begin{lemma}\label{bridge}
A connected cubic graph with a cut vertex can be drawn with the four basic slopes.
\end{lemma}
\begin{proof}
We observe that if the cubic graph has a cut vertex then it must also 
have a bridge. This bridge would be the suitable $M$-cut for using the
previous Lemma \ref{cubicdr}, since neither of the components can be disconnected
or cycles.
\end{proof}

\begin{lemma}\label{twodiscset}
A connected cubic graph with a two vertex disconnecting set can be drawn
with the four basic slopes.
\end{lemma}
\begin{proof}
If a cubic graph has a two vertex disconnecting set, then it must have a 
cut of size two with non-adjacent edges. Again the two components
we obtain must be connected (or the graph has a bridge) and cannot be cycles.
Thus we can apply Lemma \ref{cubicdr} again to get the required drawing.
\end{proof}

The following theorem was proved by Max Engelstein \cite{eng}.

\begin{lemma}\label{maxeng}
Every $3$-connected cubic graph with a Hamiltonian cycle can be drawn 
in the plane with the four basic slopes.
\end{lemma}

Note that combining the last three lemmas, %\ref{bridge,twodiscset,maxeng}
we even get

\begin{corollary} 
Every cubic graph with a Hamiltonian cycle can be drawn 
in the plane with the four basic slopes.
\end{corollary}

The graphs which now need to be checked satisfy the following conditions:
\begin{enumerate}
\item the number of vertices is at most $16$
\item the graph is $3$-connected
%\item the graph has girth at least $4$
\item the graph does not have a Hamiltonian cycle.
\end{enumerate}

% Figure by Padmini
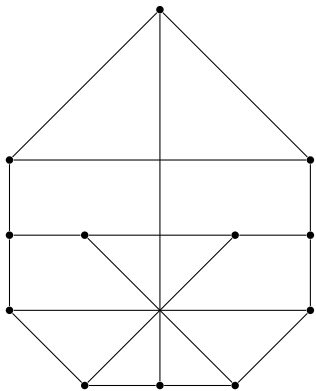
\begin{figure}[htp]\label{fig:tietze}
{\centering
\begin{tikzpicture}[scale=1]
%\filldraw [blue!80!black!20!white] (0,0) -- (1,1) -- (2.5,1) -- (3.5,0) -- (2.5,-1) -- (1,-1) -- (0,0) -- cycle;
\node [fill=black,circle,inner sep=1pt] (1) at (0,1) {}; 
\node [fill=black,circle,inner sep=1pt] (2) at (0,2) {}; 
\node [fill=black,circle,inner sep=1pt] (3) at (1,2) {}; 
\node [fill=black,circle,inner sep=1pt] (4) at (3,2) {}; 
\node [fill=black,circle,inner sep=1pt] (5) at (4,2) {}; 
\node [fill=black,circle,inner sep=1pt] (6) at (4,1) {}; 
\node [fill=black,circle,inner sep=1pt] (7) at (3,0) {}; 
\node [fill=black,circle,inner sep=1pt] (8) at (2,0) {}; 
\node [fill=black,circle,inner sep=1pt] (9) at (1,0) {}; 
\node [fill=black,circle,inner sep=1pt] (10) at (0,3) {}; 
\node [fill=black,circle,inner sep=1pt] (11) at (4,3) {}; 
\node [fill=black,circle,inner sep=1pt] (12) at (2,5) {}; 
\draw [black] (1) -- (2) -- (3) -- (4) -- (5) -- (6) -- (1);
\draw [black] (1) -- (9) -- (8) -- (7) -- (6);
\draw [black] (10) -- (11) -- (12) -- (10);
\draw [black] (9) -- (4);
\draw [black] (7) -- (3);
\draw [black] (2) -- (10);
\draw [black] (8) -- (12);
\draw [black] (5) -- (11);
%\draw [->,green] (-0.5,4) arc (145:35:22pt);
\end{tikzpicture}
\caption{The Tietze's graph drawn with the four basic slopes.}
} 
\end{figure}

Note that if the number of vertices is at most $16$, then it follows from Lemma \ref{girth} that the girth is at most $6$.
Luckily there are several lists available of cubic graphs with a given number of vertices, $n$ and a given girth, $g$.

If $g=6$, then there are only two graphs with at most $16$ vertices (see \cite{graphlist, meringer}), both containing a Hamiltonian cycle.

If $g=5$ and $n=16$, then Lemma \ref{supercycle} gives a supercycle with at most $8$ vertices, so using Lemma \ref{Mcut} we are done.

If $g=5$ and $n=14$, then there are only nine graphs (see \cite{graphlist, meringer}), all containing a Hamiltonian cycle.

If $g\le 4$ and $n=16$, then Lemma \ref{supercycle} gives a supercycle with at most $8$ vertices, so using Lemma \ref{Mcut} we are done.

If $g\le 4$ and $n=14$, then Lemma \ref{supercycle} gives a supercycle with at most $7$ vertices, so using Lemma \ref{Mcut} we are done.

Finally, all graphs with at most $12$ vertices are either not $3$-connected or contain a Hamiltonian cycle, except for the Petersen graph and Tietze's Graph (see \cite{wiki}).
For the drawing of these two graphs, see the respective Figures.
%The drawing of Petersen graph can be found in the Introduction \ref{fig:petersen} and we have given above a drawing of Tietze's graph in Figure \ref{fig:tietze}.

\section{Which four slopes? and other concluding questions}
After establishing Theorem \ref{mainthm} the question arises whether we could have used any other four slopes.
Call a set of slopes {\em good} if every cubic graph has a straight-line drawing with them. 
In this section we prove Theorem \ref{karakterizacio} that claims that the following statements are equivalent for a set $S$ of four slopes.

\begin{enumerate}
\item $S$ is good.
\item $S$ is an affine image of the four basic slopes.
\item We can draw $K_4$ with $S$.
\end{enumerate}

\begin{proof}
Since affine transformation keeps incidences, any set that is the affine image of the four basic slopes is good.

On the other hand, if a set $S=\{s_1,s_2,s_3,s_4\}$ is good, then $K_4$ has a straight-line drawing with $S$. Since we do not allow a vertex to be in the interior of an edge, the four vertices must be in general position. This implies that two incident edges cannot have the same slope. Therefore there are two slopes, without loss of generality $s_1$ and $s_2$, such that we have two-two edges of each slope. These four edges must form a cycle of length four, which means that the vertices are the vertices of a parallelogram. But in this case there is an affine transformation that takes the parallelogram to a square. This transformation also takes $S$ into the four basic slopes.
\end{proof}

Note that a similar reasoning shows that no matter how many slopes we take, their set need not be good, because we cannot even draw $K_4$ with them unless they satisfy some correlation. 
As in the proofs it is used only a few times that our slopes are the four basic slopes (for rotation invariance and to start induction), we make the following conjecture.

\begin{conjecture}
There is a (not necessarily connected, finite) graph such that a set of slopes is good if and only if this graph has a straight-line drawing with them.
\end{conjecture}

This finite graph would be the disjoint union of $K_4$, maybe the Petersen graph and other small graphs. We could not even rule out the possibility that $K_4$ (or maybe another, connected graph) is alone sufficient. Note that we can define a partial order on the graphs this way. Let $G < H$ if any set of slopes that can be used to draw $H$ can also be used to draw $G$. This way of course $G\subset H \Rightarrow G<H$ but what else can we say about this poset?

Is it possible to use this new method to prove that the slope parameter of cubic graphs is also four?

The main question remains to prove or disprove whether the slope number of graphs with maximum degree four is unbounded.

\bibliographystyle{plain}

\begin{thebibliography}{10}

\bibitem{graphlist}
http://www.mathe2.uni-bayreuth.de/markus/reggraphs.html.

\bibitem{wiki}
http://en.wikipedia.org/wiki/Table of simple cubic graphs.

\bibitem{amba06}
G.~Ambrus and J{\'a}nos Bar{\'a}t.
\newblock A contribution to queens graphs: A substitution method.
\newblock {\em Discrete Mathematics}, 306(12):1105--1114, 2006.

\bibitem{ambaha06}
G.~Ambrus, J{\'a}nos Bar{\'a}t, and P.~Hajnal.
\newblock The slope parameter of graphs.
\newblock {\em Acta Sci. Math. (Szeged)}, 72:875--889, 2006.

\bibitem{Angelini}
Patrizio Angelini, Luca Cittadini, Giuseppe {Di Battista}, Walter Didimo,
  Fabrizio Frati, Michael Kaufmann, and Antonios Symvonis.
\newblock On the perspectives opened by right angle crossing drawings.
\newblock In {\em 17th Sympos. Graph Drawing (GD'09)}, volume 5849 of {\em
  LNCS}, pages 21--32, 2010.

\bibitem{rado}
Karin Arikushi, Radoslav Fulek, Bal{\'a}zs Keszegh, Filip Moric, and Csaba~D.
  T{\'o}th.
\newblock Graphs that admit right angle crossing drawings.
\newblock In Dimitrios~M. Thilikos, editor, {\em WG}, volume 6410 of {\em
  Lecture Notes in Computer Science}, pages 135--146, 2010.

\bibitem{bmw06}
J{\'a}nos Bar{\'a}t, Jir\'{\i} Matousek, and David~R. Wood.
\newblock Bounded-degree graphs have arbitrarily large geometric thickness.
\newblock {\em Electr. J. Comb.}, 13(1), 2006.

\bibitem{bs09}
J.~Bell and B.~Stevens.
\newblock A survey of known results and research areas for $n$-queens.
\newblock {\em Discrete Mathematics}, 309:1--31, 2009.

\bibitem{Didimo}
Walter Didimo, Peter Eades, and Giuseppe Liotta.
\newblock Drawing graphs with right angle crossings.
\newblock In {\em Proc. 11th WADS}, volume 5664 of {\em LNCS}, pages 206--217.
  Springer, 2009.

\bibitem{DiEH}
Michael~B. Dillencourt, David Eppstein, and Daniel~S. Hirschberg.
\newblock Geometric thickness of complete graphs.
\newblock {\em J. Graph Algorithms Appl.}, 4(3):5--17, 2000.

\bibitem{dsw07}
Vida Dujmovic, David Eppstein, Matthew Suderman, and David~R. Wood.
\newblock Drawings of planar graphs with few slopes and segments.
\newblock {\em Comput. Geom.}, 38(3):194--212, 2007.

\bibitem{dsw04}
Vida Dujmovic, Matthew Suderman, and David~R. Wood.
\newblock Really straight graph drawings.
\newblock In J{\'a}nos Pach, editor, {\em Graph Drawing}, volume 3383 of {\em
  Lecture Notes in Computer Science}, pages 122--132. Springer, 2004.

\bibitem{DuW}
Vida Dujmovic and David~R. Wood.
\newblock Graph treewidth and geometric thickness parameters.
\newblock {\em Discrete {\&} Computational Geometry}, 37(4):641--670, 2007.

\bibitem{dek04}
Christian~A. Duncan, David Eppstein, and Stephen~G. Kobourov.
\newblock The geometric thickness of low degree graphs.
\newblock In Jack Snoeyink and Jean-Daniel Boissonnat, editors, {\em Symposium
  on Computational Geometry}, pages 340--346. ACM, 2004.

\bibitem{eng}
M.~Engelstein.
\newblock Drawing graphs with few slopes.
\newblock {\em Intel Competition for high school students}, 2005.

\bibitem{E04}
David Eppstein.
\newblock Separating thickness from geometric thickness.
\newblock In Stephen~G. Kobourov and Michael~T. Goodrich, editors, {\em Graph
  Drawing}, volume 2528 of {\em Lecture Notes in Computer Science}, pages
  150--161. Springer, 2002.

\bibitem{HuSW}
Joan~P. Hutchinson, Thomas~C. Shermer, and Andrew Vince.
\newblock On representations of some thickness-two graphs.
\newblock {\em Comput. Geom.}, 13(3):161--171, 1999.

\bibitem{JJ10}
V\'{\i}t Jel\'{\i}nek, Eva Jel\'{\i}nkov{\'a}, Jan Kratochv\'{\i}l, Bernard
  Lidick{\'y}, Marek Tesar, and Tom{\'a}s Vyskocil.
\newblock The planar slope number of planar partial 3-trees of bounded degree.
\newblock In David Eppstein and Emden~R. Gansner, editors, {\em Graph Drawing},
  volume 5849 of {\em Lecture Notes in Computer Science}, pages 304--315.
  Springer, 2009.

\bibitem{Ka}
P.~C. Kainen.
\newblock Thickness and coarseness of graphs.
\newblock {\em Abh. Math. Sem. Univ. Hamburg}, 39:88--95, 1973.

\bibitem{kpp10}
Bal{\'a}zs Keszegh, J{\'a}nos Pach, and D{\"o}m{\"o}t{\"o}r P{\'a}lv{\"o}lgyi.
\newblock Drawing planar graphs of bounded degree with few slopes.
\newblock In Ulrik Brandes and Sabine Cornelsen, editors, {\em Graph Drawing},
  volume 6502 of {\em Lecture Notes in Computer Science}, pages 293--304.
  Springer, 2010.

\bibitem{kppt08_2}
Bal{\'a}zs Keszegh, J{\'a}nos Pach, D{\"o}m{\"o}t{\"o}r P{\'a}lv{\"o}lgyi, and
  G{\'e}za T{\'o}th.
\newblock Drawing cubic graphs with at most five slopes.
\newblock {\em Comput. Geom.}, 40(2):138--147, 2008.

\bibitem{kppt10}
Bal{\'a}zs Keszegh, J{\'a}nos Pach, D{\"o}m{\"o}t{\"o}r P{\'a}lv{\"o}lgyi, and
  G{\'e}za T{\'o}th.
\newblock Cubic graphs have bounded slope parameter.
\newblock {\em J. Graph Algorithms Appl.}, 14(1):5--17, 2010.

\bibitem{meringer}
Markus Meringer.
\newblock Fast generation of regular graphs and construction of cages.
\newblock {\em J. Graph Theory}, 30:137--146, February 1999.

\bibitem{ms}
Padmini Mukkamala and Mario Szegedy.
\newblock Geometric representation of cubic graphs with four directions.
\newblock {\em Comput. Geom.}, 42(9):842--851, 2009.

\bibitem{MuOS}
Petra Mutzel, Thomas Odenthal, and Mark Scharbrodt.
\newblock The thickness of graphs: A survey.
\newblock {\em Graphs Combin}, 14:59--73, 1998.

\bibitem{pp06}
J{\'a}nos Pach and D{\"o}m{\"o}t{\"o}r P{\'a}lv{\"o}lgyi.
\newblock Bounded-degree graphs can have arbitrarily large slope numbers.
\newblock {\em Electr. J. Comb.}, 13(1), 2006.

\bibitem{wc94}
Greg~A. Wade and Jiang-Hsing Chu.
\newblock Drawability of complete graphs using a minimal slope set.
\newblock {\em Comput. J.}, 37(2):139--142, 1994.

\end{thebibliography}

\appendix
\section{Program code}
The following code is in Maple.

\begin{verbatim}
#For accessing log, ceil functions.
with(MTM);

#fmax is a procedure that computes the girth for which a graph on N
#vertices will have the largest supercyle.
#Here, mg denotes the maximum possible girth, max and g will have the
#values of the maximum size of the supercycle and the girth at which 
#it occurs respectively. The procedure returns 2s-2, if this value is
#less than N, we can apply Lemma 2.6 and 2.8 to draw the graphs on N
#vertices.
fmax := proc (N) local g, mg, max, i, exp; 

#Initializations
max := -1;
g := 0;
mg := 2*ceil(evalf(log2((1/3)*N+1))); 

if mg < 3 then RETURN([N, 2*max-2, mg, g]) fi; 

#Main search cycle.
for i from 3 while i <= mg do 
     exp := 2*ceil(evalf(log2((N+1)/i)))+i-1; 
     if max < exp then max := exp; g := i fi
end do; 

RETURN([N, 2*max-2, mg, g]) 
end proc;

seq(fmax(i), i = 6 .. 42, 2);
[6,10,4,3], [8,12,4,4], [10,14,6,5], [12,16,6,6], [14,16,6,6],
[16,16,6,4], [18,16,6,4], [20,18,6,5], [22,20,8,8], [24,20,8,6],
[26,20,8,6], [28,22,8,7], [30,22,8,7], [32,24,8,8], [34,24,8,8], 
[36,24,8,8], [38,24,8,8], [40,24,8,8], [42,24,8,8]
\end{verbatim}

\end{document}